\documentclass[journal, draft,onecolumn]{IEEEtran}
\usepackage[utf8]{inputenc} 
\usepackage[T1]{fontenc}
\usepackage{url}
\usepackage{ifthen}
\usepackage{cite}
\usepackage[cmex10]{amsmath} 

\usepackage{todonotes}
\usepackage{amssymb}
\usepackage{verbatim}
\usepackage{bm}
\usepackage{color,graphicx,xcolor}
\usepackage{mdwtab}
\usepackage{subfigure}
\usepackage{mathtools,tikz}
\usepackage{hhline}
\usepackage{multirow}
\usepackage{pdfpages}
\usepackage{enumitem}
\usepackage{balance}

\usepackage{cite}
\usepackage{amsfonts,xspace}
\usepackage{algorithmic}
\usepackage{graphicx}
\usepackage{textcomp}

\usepackage{algorithmic}
\usepackage{textcomp}
\usepackage{amsthm}

\usepackage{pgfplots}
\pgfplotsset{compat=1.15}
\usepackage{mathrsfs}
\usetikzlibrary{arrows}
\usepackage{adjustbox}
\usepackage{booktabs}
\usepackage{authblk}

\newcommand{\blue}{\textcolor{black}}

\DeclareMathOperator*{\argmin}{arg\,min}

\newtheorem{theorem}{Theorem}

\newtheorem{lemma}{Lemma}

\begin{document}
\title{On Robust Hypothesis Testing with respect to the Hellinger Distance}
	
\author[]{Eeshan Modak}
\affil[]{Tata Institute of Fundamental Research, Mumbai, India}

\author[]{Sivaraman Balakrishnan}
\affil[]{Carnegie Mellon University}

\author[]{Ananda Theertha Suresh}
\affil[]{Google Research, New York}

\maketitle



\begin{abstract}
    We study a variant of the simple hypothesis testing problem where observed samples do not necessarily come from either of the specified distributions, but rather from a close variant of them. In this setting, we require a test that is robust to misspecification and identifies which distribution is closer in Hellinger distance. If the underlying distribution is nearly equidistant from both hypotheses, the problem becomes intractable. Our main result is a lower bound on the slack factor, which quantifies how much closer the underlying distribution must be to one hypothesis relative to the other for any test to remain robust. We also demonstrate the implications of this result for testing with respect to symmetric chi-squared distance. Finally, we study an alternative way to specify robustness, where each hypothesis is a Hellinger ball around a fixed distribution. We provide and analyze a test for this composite hypothesis testing problem.
\end{abstract}

\section{Introduction}
Hypothesis testing is a classical problem in statistics and has been studied for more than a century. In its simplest form, we have the following problem: Given two distributions $p_0$ and $p_1$ over some domain $\mathcal{X}$, we observe $m$ independent and identically distributed (i.i.d.) samples $X_1, \ldots, X_m$ from a distribution $p \in \{p_0,p_1\}$. We have to decide whether $p$ is $p_0$ or $p_1$.
\begin{align*}
    H_0 &: p = p_0 \\
    H_1 &: p = p_1.
\end{align*}
There are two types of error that can occur. The Type-I error is the probability of declaring $H_1$ under $H_0$ and the Type-II error is the probability of declaring $H_0$ under $H_1$. Naturally, there is a tradeoff between the two errors. The Neyman-Pearson test achieves the optimal tradeoff. It has the following form: Fix some threshold $t$ and declare $H_0$ if the likelihood ratio $\frac{p_0(X^m)}{p_1(X^m)} \ge t$, declare $H_1$ otherwise.

\subsection{Robust Hypothesis Testing}
In many real-world contexts, the assumption that $p \in \{p_0,p_1\}$ might be too restrictive. This could be due to noise in the sampling process or an imperfection in the modeling.
The goal of robust testing is to be resistant to violations of such normative assumptions. In composite hypothesis testing, each hypothesis is associated with a set of distributions (say $\mathcal{P}_0$ under $H_0$ and $\mathcal{P}_1$ under $H_1$) and samples are i.i.d. according to some fixed distribution from the set. For distributions $p_0' \in \mathcal{P}_0$, $p_1' \in \mathcal{P}_1$ and a test $T$, let $e(T,p_0',p_1')$ be the maximum of type I and type II error. We want a test that is $\min \max$ optimal, i.e., $\inf_{T} \sup_{p_0'\in \mathcal{P}_0,p_1'\in \mathcal{P}_1} e(T,p_0',p_1')$.
Huber~\cite{huber1965robust} considered the following model: $\mathcal{P}_0= \{p: p=(1-\epsilon) p_0 + \epsilon q, q \in \Delta_{\mathcal{X}}\}$, $\mathcal{P}_1= \{p: p=(1-\epsilon) p_1 + \epsilon q, q \in \Delta_{\mathcal{X}}\}$
where $\epsilon$ is some small positive constant and $\Delta_{\mathcal{X}}$ is the set of distributions over $\mathcal{X}$. \blue{They show that the optimal test is a clipped likelihood ratio test given by $\max\{c_0, \min \{ c_1, \frac{p_1(X)}{p_0(X)} \} \}$ where $c_0,c_1$ depend on $p_0,p_1$ and $\epsilon$.} Levy~\cite{levy2008robust}, G{\"u}l and Zoubir~\cite{gul2017minimax} study models where $\mathcal{P}_0 = \{p: D(p\|p_0) \le \epsilon, p \in \Delta_{\mathcal{X}} \}$ and $\mathcal{P}_1 = \{p: D(p\|p_1) \le \epsilon, p \in \Delta_{\mathcal{X}} \}$ (here, $D(p\|p_0)$ denotes the Kullback-Leibler (KL) divergence between $p$ and $p_0$ given by $\int_{x}p(x) \log \frac{p(x)}{p_0(x)} dx )$.
\subsection{Robustness to imperfect modelling}
In many practical settings, our hypotheses are merely our models for the observed data. However, the true distribution could be different from both hypotheses. Consider the following problem: Say we have the following hypotheses to model the number of applicants for a particular job: 
\begin{align*}
    H_0 &: \text{Pois}(\lambda_0) \\
    H_1 &: \text{Pois}(\lambda_1).
\end{align*}
Here, $\text{Pois}(\lambda)$ denotes the Poisson distribution with parameter $\lambda$. Now, the actual distribution of the number of applicants could be some unknown complicated distribution. In such a case, we would like our test to output the hypothesis that is closer (with respect to some distance metric) to the true distribution. Thus, we can formulate the problem as follows. Let $p_0$ and $p_1$ be the two model distributions and $p$ be the true distribution \blue{(all defined on a common domain $\mathcal{X}$). Let $d(.,.)$
be a distance metric on the space of distributions on $\mathcal{X}$}. After observing i.i.d. samples from $p$, we need to distinguish between
\begin{align*}
    H_0 &: d(p,p_0) < d(p,p_1) \\
    H_1 &: d(p,p_0) > d(p,p_1).
\end{align*}
If the distribution $p$ is arbitrarily close to being equidistant from $p_0$ and $p_1$, then it is impossible to output the closer distribution using finitely many samples. Hence, we introduce a slack factor $\gamma > 1$ and reformulate the problem.
\begin{align*}
    H_0 &: \gamma d(p,p_0) < d(p,p_1) \\
    H_1 &: d(p,p_0) > \gamma d(p,p_1).
\end{align*}
If $p$ is not in $H_0$ or in $H_1$, then the test can make any decision. The error is defined to be the maximum of type-I and type-II errors. We call a test $\gamma$-robust if it uses finitely many samples and keeps the probability of error below some fixed constant $\delta<\frac{1}{2}$ for any choice of $p_0$, $p_1$, and $p$. Let $\gamma^*$ be the smallest $\gamma$ for which a $\gamma$-robust test exists. We call $\gamma^*$ the optimal slack factor for the robust testing problem. Previous works have tried to characterize $\gamma^*$ for different distance measures $d(.,.)$. The $\ell_p$-norm of a function $f:\mathcal{X} \rightarrow \mathbb{R}$ is given by
\begin{align*}
    \|f\|_{p} := \left( \int_{x \in \mathcal{X}} |f(x)|^{p} dx \right)^{\frac{1}{p}}.
\end{align*}
The total variation (TV) distance two distributions $p_0$ and $p_1$ over $\mathcal{X}$ is given by
\begin{align*}
  TV(p_0,p_1) &= \frac{1}{2} \|p_0(x) - p_1(x)\|_1  \\
  &= \sup_{A} |p_0(A)-p_1(A)|.
\end{align*}
When the distance measure $d(.,.)$ is Total Variation (TV), it is known that $\gamma^* = 3$. The best test is based on the Scheff\'e estimator \cite[Theorem 6.1] {devroye2001combinatorial} which is given by 
\begin{equation*}
    T(X^m) = \begin{cases} 
          H_0  & \text{if} \hspace{5pt} |p_0(A) -  \mu(A)| \leq |p_1(A) -  \mu(A)| \\
          H_1 & \text{else.}
       \end{cases}
\end{equation*}
where $A = \{ x : p_0(x) > p_1(x) \}$ and $\mu(A) = \frac{1}{m} \sum_{i=1}^{m} 1[X_i \in A] $, that is, the fraction of observed samples which fall in the set $A$. The lower bound ($\gamma^* \ge 3$) was given in the work of Bousquets, Kane, and Moran \cite{bousquet2019optimal}. \blue{Mahalanabis and Stefankovic~\cite{mahalanabis2007density} have shown that $\gamma^*=2$ when the test is randomized and $TV(p,p_0)$ (resp. $TV(p,p_1)$) is replaced by $\mathbb{E}[TV(p,p_0)]$ (resp. $\mathbb{E}[TV(p,p_1)]$) in the hypothesis formulation.} The Hellinger distance between $p_0$ and $p_1$ is given by
\begin{align*}
    H^2(p_0,p_1) &= \frac{1}{2} {\|\sqrt{p_0}-\sqrt{p_1}\|}_2^2 \\
    &= 1 - \int_{x \in \mathcal{X}} \sqrt{p_0(x)p_1(x)}dx.
\end{align*}
It can be shown that $\frac{1}{2}TV^2(p_0,p_1) \le H^2(p_0,p_1) \le TV(p_0,p_1)$. Hellinger distance has some interesting properties: (i) It has a tensorization property (to decompose the distance between product distributions), and (ii) It is related to the notion of fidelity in the quantum information literature through the Bhattacharya distance, $B(p_0,p_1) = \int_{x \in \mathcal{X}} \sqrt{p_0(x)p_1(x)}dx$. Robust testing with respect to the Hellinger distance has been studied in the works of \cite{suresh2021robust,baraud2011estimator}.
Suresh \cite{suresh2021robust} constructed a test that worked for as long as $\gamma > (\frac{\sqrt{2}}{\sqrt{2}-1})^2$. Baraud's test \cite{baraud2011estimator} worked for $\gamma > \frac{\sqrt{2}+1}{\sqrt{2}-1}$. Thus, we know that $\gamma^* \le \frac{\sqrt{2}+1}{\sqrt{2}-1}$. However, a lower bound was not known in this case. In this work, we make some progress on this front.

\subsection{Our work}
We show that the optimal slack factor $\gamma^*$ is at least $\frac{\sqrt{2}}{\sqrt{2}-1}$ for both deterministic and randomized tests. Thus, there is a gap between the upper and lower bound. However, we show that our lower bound is tight under the constraint that the distributions in our model class have disjoint supports, that is, $p_0 \bot p_1$. We also show that the optimal slack factor $\gamma^*$ for the problem with respect to the symmetric chi-squared distance is at least $3$. Finally, we show that a simple modification of Baraud's test can be used to solve the composite hypothesis testing problem where each set is a ball of radius $r$ (in Hellinger distance) around $p_0$ or $p_1$.


\subsection{Notation and Convention}
In the remainder of the paper, $p_0$, $p_1$, and $p$ will be distributions over some domain $\mathcal{X}$. If $\mathcal{X}$ is discrete, then they will be probability mass functions (p.m.f.). If $\mathcal{X}$ is continuous ($\mathbb{R}^d$ in our case), then they will be densities with respect to the Lebesgue measure. \blue{For simplicity of exposition, we will treat $\mathcal{X}$ as discrete in Sections~\ref{sec:upper_bound} and \ref{sec:alt_test}. However, all arguments go through in the continuous case by replacing p.m.f. with densities and sums with integrals.} For any integer $N$, $[N]$ denotes the set $\{1, 2, \ldots, N\}$. Let $\text{unif}[a,b]$ denote the uniform distribution over the interval $[a,b]$.

\section{Problem Setup} \label{sec:problem_setup}
Let $\Delta(\mathcal{X})$ be the set of all probability distributions over $\mathcal{X}$. Let $\mathcal{P}=\{p_0,p_1\}$ be a model class where $p_0,p_1 \in \Delta(\mathcal{X})$. Let $p \in \Delta(\mathcal{X})$ be a target distribution. Upon receiving $m$ i.i.d. samples from $p$, consider the following hypothesis testing problem.
\begin{align*}
    H_0 &:\gamma H^2(p,p_0) \le H^2(p,p_1) \\
    H_1 &: H^2(p,p_0) \ge \gamma H^2(p,p_1).
\end{align*}
We are promised that the target distribution $p$ belongs to one of the hypotheses. Define the probability of error to be the maximum of type I and type II errors. We say $\mathcal{P}=\{p_0,p_1\}$ is $\gamma$-robust testable if there exists a (possibly randomized) test $T$ such that for every $\delta>0$ there is a finite sample complexity bound $m=m(\delta)$ such that for every target distribution $p$, if $T$ receives at least $m$ i.i.d. samples from $p$, it outputs $\argmin \limits_{q \in \mathcal{P}} d(q,p)$.  with probability at least $1-\delta$. We will fix $\delta$ to be equal to $\frac{1}{3}$ for the exposition. Let $\gamma^*$ denote the optimal slack factor, i.e. the smallest $\gamma$ for which every class $\mathcal{P}$ is $\gamma$-robust testable. To get an upper bound on $\gamma^*$, we need to construct a test which can $\gamma$-robustly test all classes $\mathcal{P}$. To obtain a lower bound on $\gamma^*$, we need to show that a particular class $\mathcal{P}$ is not $\gamma$-robustly testable.
\section{Upper Bound on $\gamma^*$} \label{sec:upper_bound}
What kinds of tests are $\gamma$-robust for some $\gamma > 1$? It is easy to see that the maximum likelihood (ML) test is not suitable for this problem in general. Consider the following example. Let $\text{unif[a,b]}$ denote the uniform distribution over the interval $[a,b]$. Let $p_0=\text{unif}[-1,1]$, $p_1=\text{unif}[\epsilon,1+\epsilon]$ and $p=\text{unif}[0,1]$. Observe that $H^2(p,p_0) = \frac{\sqrt{2}}{\sqrt{2}-1}$ and $H^2(p,p_1) = \epsilon$. Thus, $p_1$ can be arbitrarily closer to $p$ than $p_0$. After observing $m$ i.i.d. samples from $p$, the ML test will output $\arg \max \left\{ \prod_{i=1}^{m} p_0(X_i), \prod_{i=1}^{m} p_1(X_i) \right\}$. The test will output $p_1$ only if none of the $m$ samples falls in the interval $[0,\epsilon]$. The probability of error is equal to $1-(1-\epsilon)^m$ which tends to $1$ as $m \rightarrow \infty$. Thus, the ML test is not $\gamma$-robust for any $\gamma > 1$.

The tests suitable for the robust testing problem involve appropriate bounded proxies for the logarithm, i.e., some function $\psi$ such that $\psi(x) \in [-1,1]$ for all $x$ and \blue{$\psi(\frac{1}{x})=-\psi(x)$}. The tests in \cite{suresh2021robust,baraud2011estimator} fit in this framework. We briefly describe the test constructed by Baraud \cite{baraud2011estimator}, thereby showing an upper bound on $\gamma^*$.
\begin{theorem}[Baraud \cite{baraud2011estimator}]
    For $\gamma \ge \frac{\sqrt{2}+1}{\sqrt{2}-1}$, every class $\mathcal{P}=\{p_0,p_1\}$ is $\gamma$-robustly testable.
\end{theorem}
\begin{proof}
Let $q:= \frac{p_0+p_1}{2}$. Let the test statistic be given by
    \begin{equation} \label{eqn:baraud_test}
        T(X^m) = \frac{1}{m} \sum_{i=1}^{m} \left[ \sqrt{\frac{p_0(X_i)}{q(X_i)}} - \sqrt{\frac{p_1(X_i)}{q(X_i)}}  \right] + H^2(p_1,q) - H^2(p_0,q).
    \end{equation}
    We will show that $\mathbb{E}[T(^m)] \ge 0$ under $H_0$ and $\mathbb{E}[T(^m)] \le 0$ under $H_1$. By linearity of expectation, it suffices to analyse for $m=1$. We will first analyze the expected value of $T(X)$ under $H_0$.
    \begin{align*}
        \mathbb{E}[(X)] &= \sum_{x}  p(x)\sqrt{\frac{p_0(x)}{q(x)}} -2\sqrt{p(x)p_0(x)} + 2\sqrt{p(x)p_0(x)} + \sqrt{p_0(x)q(x)} \\
        & \hspace{20pt} - p(x)\sqrt{\frac{p_1(x)}{q(x)}} + 2\sqrt{p(x)p_1(x)} - 2\sqrt{p(x)p_1(x)} - \sqrt{p_1(x)q(x)} \\
        &= 2H^2(p,p_1)-2H^2(p,p_0) + \sum_{x} \sqrt{\frac{p_0(x)}{q(x)}} (\sqrt{p(x)}-\sqrt{q(x)})^2 - \sum_{x} \sqrt{\frac{p_1(x)}{q(x)}} (\sqrt{p(x)}-\sqrt{q(x)})^2 \\
        &\overset{(a)}{\ge} 2H^2(p,p_1)-2H^2(p,p_0) - \sum_{x} \sqrt{\frac{p_1(x)}{q(x)}} (\sqrt{p(x)}-\sqrt{q(x)})^2 \\
        &\overset{(b)}{\ge} 2H^2(p,p_1)-2H^2(p,p_0) - \sqrt{2}\sum_{x} (\sqrt{p(x)}-\sqrt{q(x)})^2 \\
        &= 2H^2(p,p_1)-2H^2(p,p_0) - 2\sqrt{2} H^2(p,q) \\
        &\overset{(c)}{\ge} 2H^2(p,p_1)-2H^2(p,p_0) - 2\sqrt{2} \left( \frac{H^2(p,p_0)+H^2(p,p_1)}{2} \right) \\
        &= (2-\sqrt{2}) H^2(p,p_1) - (2 + \sqrt{2}) H^2(p,p_0).
    \end{align*}
    The first inequality $(a)$ is obtained by dropping a non-negative term, $(b)$ follows from the fact that $\frac{p_1(x)}{p_0(x)+p_1(x)} \le 1$, $(c)$ follows from the convexity of $H^2(.,.)$ (and thus $H^2\left(p,\frac{p_0+p_1}{2} \le \frac{H^2(p,p_0)+H^2(p,p_1)}{2} \right)$).
    Thus, $\mathbb{E}[T(x)] \ge 0$ if $H^2(p,p_1) \ge \frac{\sqrt{2}+1}{\sqrt{2}-1} H^2(p,p_0)$. Likewise, we can show that $\mathbb{E}[T(x)] \le 0$ if $H^2(p,p_0) \ge \frac{\sqrt{2}+1}{\sqrt{2}-1} H^2(p,p_1)$. \blue{Since the test statistic is bounded, we can use Hoeffding's inequality to bound the error.} This completes the proof. 
\end{proof}
\section{Lower Bound on $\gamma^*$} \label{sec:lower_bound}
We now show that $\gamma^* \ge \frac{\sqrt{2}}{\sqrt{2}-1}$. The construction given in \cite[Section 4]{bousquet2019optimal} does not work in our case. We have to make certain adaptations. However, we follow their general technique which constructs a contradiction using Le Cam style argument.
\begin{theorem} \label{thrm:robust_lb}
    For every $\gamma < \frac{\sqrt{2}}{\sqrt{2}-1}$, there is a class $\mathcal{P}=\{p_0,p_1\}$ which is not $\gamma$-robustly testable.
\end{theorem}
We first outline the sketch of the proof. \blue{We first fix $p_0$ and $p_1$. Let $m$ be the number of i.i.d. samples drawn from the true distribution $p$ . For any fixed $m$, we have the following.}
\begin{enumerate}
    \item \blue{We construct a family of distributions $\mathcal{D}_0$ such that for every $p \in \mathcal{D}_0$ and any arbitrarily small $\epsilon$, we have
    \begin{align*}
    \frac{\sqrt{2}}{\sqrt{2}-1} - \epsilon < \frac{H^2(p,p_1)}{H^2(p,p_0)} < \frac{\sqrt{2}}{\sqrt{2}-1}.    
    \end{align*}}
    \blue{We construct another family of distributions $\mathcal{D}_1$ such that for every $p \in \mathcal{D}_1$ and any arbitrarily small $\epsilon$, we have
    \begin{align*}
     \frac{\sqrt{2}}{\sqrt{2}-1} - \epsilon < \frac{H^2(p,p_0)}{H^2(p,p_1)} < \frac{\sqrt{2}}{\sqrt{2}-1}.   
    \end{align*}}
    \item Let $\mathcal{D}_0^m$ denote the following product distribution: pick $p$ uniformly at random from the family $\mathcal{D}_0$ and draw m i.i.d. samples from $p$. Likewise, $\mathcal{D}_1^m$ denotes the following product distribution: pick $p$ uniformly at random from the family $\mathcal{D}_1$ and draw m i.i.d. samples from $p$. We show that $TV(\mathcal{D}_0^m,\mathcal{D}_1^m) \le \frac{1}{3}$.
    \item We use Le Cam's argument to show that the probability of error in distinguishing between $\mathcal{D}_0^m$ and $\mathcal{D}_1^m$ is at least $\frac{1}{3}$.  
    \item The above three points imply that if we had a $\gamma$-robust test with $\gamma  < \frac{\sqrt{2}}{\sqrt{2}-1}$ then we could use it to distinguish between $\mathcal{D}_0^m$ and $\mathcal{D}_1^m$ with probability of error at most $\frac{1}{3}$ which contradicts point (3). Hence, such a test cannot exist.
\end{enumerate}

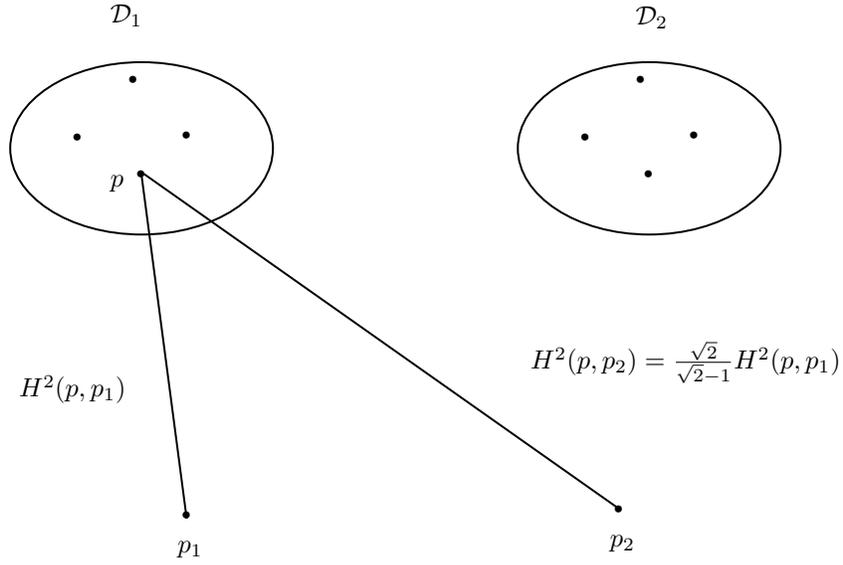
\begin{figure}
    \centering

    \tikzset{every picture/.style={line width=0.75pt}} 

\begin{tikzpicture}[x=0.75pt,y=0.75pt,yscale=-1,xscale=1]

\draw   (115,131.5) .. controls (115,107.48) and (144.66,88) .. (181.25,88) .. controls (217.84,88) and (247.5,107.48) .. (247.5,131.5) .. controls (247.5,155.52) and (217.84,175) .. (181.25,175) .. controls (144.66,175) and (115,155.52) .. (115,131.5) -- cycle ;
\draw   (371,131.5) .. controls (371,107.48) and (400.66,88) .. (437.25,88) .. controls (473.84,88) and (503.5,107.48) .. (503.5,131.5) .. controls (503.5,155.52) and (473.84,175) .. (437.25,175) .. controls (400.66,175) and (371,155.52) .. (371,131.5) -- cycle ;
\draw    (181,143.4) -- (204,318.4) ;
\draw    (181,143.4) -- (422,313.4) ;

\draw (198,328.4) node [anchor=north west][inner sep=0.75pt]    {$p_{0}$};
\draw (416,325.4) node [anchor=north west][inner sep=0.75pt]    {$p_{1}$};
\draw (200,313.4) node [anchor=north west][inner sep=0.75pt]   {\textbullet};
\draw (418,310.4) node [anchor=north west][inner sep=0.75pt]    {\textbullet};
\draw (145,122.4) node [anchor=north west][inner sep=0.75pt]    {\textbullet};
\draw (173,93.4) node [anchor=north west][inner sep=0.75pt]    {\textbullet};
\draw (177,141.4) node [anchor=north west][inner sep=0.75pt]    {\textbullet};
\draw (200,121.4) node [anchor=north west][inner sep=0.75pt]    {\textbullet};
\draw (401,122.4) node [anchor=north west][inner sep=0.75pt]    {\textbullet};
\draw (429,93.4) node [anchor=north west][inner sep=0.75pt]    {\textbullet};
\draw (433,141.4) node [anchor=north west][inner sep=0.75pt]    {\textbullet};
\draw (456,121.4) node [anchor=north west][inner sep=0.75pt]    {\textbullet};
\draw (164,57.4) node [anchor=north west][inner sep=0.75pt]    {$\mathcal{D}_{0}$};
\draw (429,58.4) node [anchor=north west][inner sep=0.75pt]    {$\mathcal{D}_{1}$};
\draw (118,244.4) node [anchor=north west][inner sep=0.75pt]    {$H^{2}( p,p_{0})$};
\draw (164,143.4) node [anchor=north west][inner sep=0.75pt]    {$p$};
\draw (376,227.4) node [anchor=north west][inner sep=0.75pt]    {$H^{2}( p,p_{1}) =\frac{\sqrt{2}}{\sqrt{2} -1} H^{2}( p,p_{0})$};

\end{tikzpicture}
    
    \caption{$\mathcal{D}_0$ is a family of distributions such that all its members are $\frac{\sqrt{2}}{\sqrt{2}-1}$ times farther to $p_1$ than to $p_0$ in Hellinger distance. Likewise, all the members of $\mathcal{D}_1$ are $\frac{\sqrt{2}}{\sqrt{2}-1}$ times farther to $p_0$ than to $p_1$ in Hellinger distance.}
    \label{fig:proof_sketch_lb}
\end{figure}

\begin{proof}
    We first construct the families $\mathcal{D}_0$ and $\mathcal{D}_1$ described in the proof sketch. Let $0 < b \le 1$. We will set the value of $b$ later. Define distribution $p_0$ as follows.
\begin{equation*}
    p_0(x) = \begin{cases} 
          1-b & x\leq 0.5 \\
          1+b & 0.5 < x \leq 1. 
       \end{cases}
\end{equation*}
Define distribution $p_1$ as follows.
\begin{equation*}
    p_1(x) = \begin{cases} 
          1+b & x\leq 0.5 \\
          1-b & 0.5 < x \leq 1 
       \end{cases}
\end{equation*}
Divide $[0,1]$ interval into $2N_m$ bins of equal size. Let $I_k=\left[\frac{k-1}{2N_m},\frac{k}{2N_m}\right)$ be the $k^{\text{th}}$ bin. Let $a_1 \ge b$ and $b \le a_2 \le 1+b$. We select a subset $R_1 \subseteq [N_m]$ of size $\left({\frac{b}{a_1}}\right)N_m$. We also select a subset $R_2 \subseteq [N_m]$ of size $\left(\frac{b}{a_2}\right)N_m$. Consider the distribution $p^{R_1,R_2}$ obtained by perturbing $p_0$.
\begin{equation*}
    p^{R_1,R_2}(x) = \begin{cases} 
          1-b & x \in I_j,j \notin R_1,j \le N_m \\
          1-b+a_1 & x \in I_j,j \in R_1,j \le N_m \\
          1+b & x \in I_j,j-N_m \notin R_2,j > N_m \\
          1+b-a_2 & x \in I_j,j-N_m \in R_2,j > N_m. 
       \end{cases}
\end{equation*}
Note that this is indeed a distribution since the probability mass we add is $a_1 \left( \frac{1}{2N_m} \right) \left( \frac{b}{a_1} N_m \right) = 0.5b$ and the probability mass we remove is $a_2 \left( \frac{1}{2N_m} \right) \left( \frac{b}{a_2} N_m \right) = 0.5b$. \\
Let $\mathcal{D}_0=\left\{ p^{R_1,R_2}: R_1,R_2 \subseteq [N_m], |R_1|=\left({\frac{b}{a_1}}\right)N_m ,|R_2|=\left({\frac{b}{a_2}}\right)N_m \right\}$ be the family of distributions (parameterized by the sets $R_1,R_2$) obtained by perturbing $p_0$. Now, consider the distribution
$\bar{p}^{R_1,R_2}$ obtained by perturbing $p_1$.
\begin{equation*}
    \bar{p}^{R_1,R_2}(x) = \begin{cases} 
          1+b & x \in I_j,j \notin R_1,j \le N_m \\
          1+b-a_2 & x \in I_j,j \in R_1,j \le N_m \\
          1-b & x \in I_j,j-N_m \notin R_2,j > N_m \\
          1-b+a_1 & x \in I_j,j-N_m \in R_2,j > N_m. 
       \end{cases}
\end{equation*}
Let $\mathcal{D}_1=\left\{ \bar{p}^{R_1,R_2}: R_1,R_2 \subseteq [N_m], |R_1|=\left({\frac{b}{a_1}}\right)N_m ,|R_2|=\left({\frac{b}{a_2}}\right)N_m \right\}$ be the family of distributions (parameterized by the sets $R_1,R_2$) obtained by perturbing $p_1$. \\

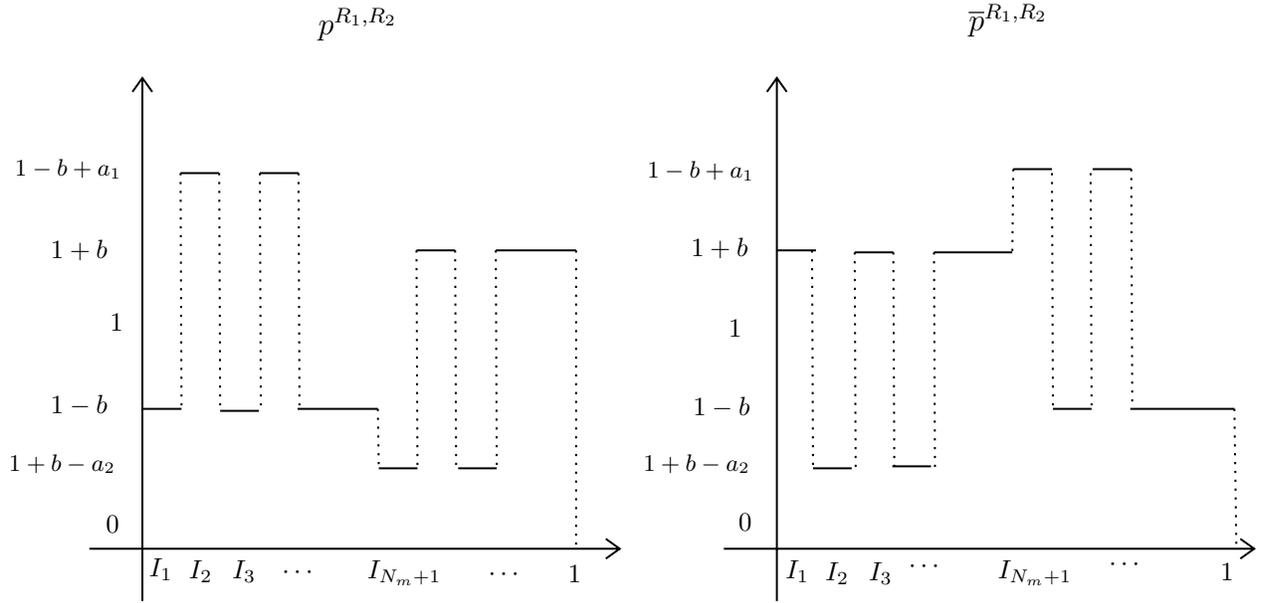
\begin{figure}
    \centering

\tikzset{every picture/.style={line width=0.75pt}} 

\begin{tikzpicture}[x=0.75pt,y=0.75pt,yscale=-1,xscale=1]

\draw  (45,301.6) -- (312.5,301.6)(71.75,64) -- (71.75,328) (305.5,296.6) -- (312.5,301.6) -- (305.5,306.6) (66.75,71) -- (71.75,64) -- (76.75,71)  ;
\draw  (365,301.6) -- (632.5,301.6)(391.75,64) -- (391.75,328) (625.5,296.6) -- (632.5,301.6) -- (625.5,306.6) (386.75,71) -- (391.75,64) -- (396.75,71)  ;
\draw    (72,231) -- (91.5,231) ;
\draw    (91,112) -- (110.5,112) ;
\draw    (111,232) -- (130.5,232) ;
\draw    (131,112) -- (150.5,112) ;
\draw    (150,231) -- (190.5,231) ;
\draw    (191,261) -- (210.5,261) ;
\draw    (210,151) -- (229.5,151) ;
\draw    (231,261) -- (250.5,261) ;
\draw    (250,151) -- (290.5,151) ;
\draw  [dash pattern={on 0.84pt off 2.51pt}]  (91,112) -- (91.5,231) ;
\draw  [dash pattern={on 0.84pt off 2.51pt}]  (110.5,112) -- (111,231) ;
\draw  [dash pattern={on 0.84pt off 2.51pt}]  (131,112) -- (131.5,231) ;
\draw  [dash pattern={on 0.84pt off 2.51pt}]  (150.5,112) -- (151,231) ;
\draw  [dash pattern={on 0.84pt off 2.51pt}]  (250,151) -- (250.5,261) ;
\draw  [dash pattern={on 0.84pt off 2.51pt}]  (229.5,151) -- (230,261) ;
\draw  [dash pattern={on 0.84pt off 2.51pt}]  (210,151) -- (210.5,261) ;
\draw  [dash pattern={on 0.84pt off 2.51pt}]  (190.5,231) -- (191,261) ;
\draw    (392,151) -- (411.5,151) ;
\draw    (410,261) -- (429.5,261) ;
\draw    (431,152) -- (450.5,152) ;
\draw    (450,260) -- (469.5,260) ;
\draw    (471,152) -- (510.5,152) ;
\draw    (511,110) -- (530.5,110) ;
\draw    (531,231) -- (550.5,231) ;
\draw    (551,110) -- (570.5,110) ;
\draw    (570,231) -- (622.5,231) ;
\draw  [dash pattern={on 0.84pt off 2.51pt}]  (530.5,112) -- (531,231) ;
\draw  [dash pattern={on 0.84pt off 2.51pt}]  (550,112) -- (550.5,231) ;
\draw  [dash pattern={on 0.84pt off 2.51pt}]  (570.5,110) -- (571,229) ;
\draw  [dash pattern={on 0.84pt off 2.51pt}]  (409.5,151) -- (410,261) ;
\draw  [dash pattern={on 0.84pt off 2.51pt}]  (431,152) -- (431.5,262) ;
\draw  [dash pattern={on 0.84pt off 2.51pt}]  (450.5,152) -- (451,262) ;
\draw  [dash pattern={on 0.84pt off 2.51pt}]  (471,152) -- (471.5,262) ;
\draw  [dash pattern={on 0.84pt off 2.51pt}]  (511,110) -- (510.5,152) ;
\draw  [dash pattern={on 0.84pt off 2.51pt}]  (290.5,151) -- (290.5,301) ;
\draw  [dash pattern={on 0.84pt off 2.51pt}]  (622.5,231) -- (623.5,301) ;

\draw (24,222.4) node [anchor=north west][inner sep=0.75pt]    {$1-b$};
\draw (24,144.4) node [anchor=north west][inner sep=0.75pt]    {$1+b$};
\draw (54,181.4) node [anchor=north west][inner sep=0.75pt]    {$1$};
\draw (6,103.4) node [anchor=north west][inner sep=0.75pt]  [font=\small]  {$1-b+a_{1}$};
\draw (3,252.4) node [anchor=north west][inner sep=0.75pt]  [font=\small]  {$1+b-a_{2}$};
\draw (159,27.4) node [anchor=north west][inner sep=0.75pt]  [font=\large]  {$p^{R_{1} ,R_{2}}$};
\draw (487,25.4) node [anchor=north west][inner sep=0.75pt]  [font=\large]  {$\overline{p}^{R_{1} ,R_{2}}$};
\draw (325,104.4) node [anchor=north west][inner sep=0.75pt]  [font=\small]  {$1-b+a_{1}$};
\draw (323,252.4) node [anchor=north west][inner sep=0.75pt]  [font=\small]  {$1+b-a_{2}$};
\draw (347,143.4) node [anchor=north west][inner sep=0.75pt]    {$1+b$};
\draw (348,223.4) node [anchor=north west][inner sep=0.75pt]    {$1-b$};
\draw (366,184.4) node [anchor=north west][inner sep=0.75pt]    {$1$};
\draw (285,308.4) node [anchor=north west][inner sep=0.75pt]    {$1$};
\draw (614,307.4) node [anchor=north west][inner sep=0.75pt]    {$1$};
\draw (52,283.4) node [anchor=north west][inner sep=0.75pt]    {$0$};
\draw (371,282.4) node [anchor=north west][inner sep=0.75pt]    {$0$};
\draw (73.75,305) node [anchor=north west][inner sep=0.75pt]    {$I_{1}$};
\draw (93.75,306) node [anchor=north west][inner sep=0.75pt]    {$I_{2}$};
\draw (115.75,306) node [anchor=north west][inner sep=0.75pt]    {$I_{3}$};
\draw (394.75,306) node [anchor=north west][inner sep=0.75pt]    {$I_{1}$};
\draw (414.75,307) node [anchor=north west][inner sep=0.75pt]    {$I_{2}$};
\draw (436.75,307) node [anchor=north west][inner sep=0.75pt]    {$I_{3}$};
\draw (183.75,306) node [anchor=north west][inner sep=0.75pt]    {$I_{N_{m} +1}$};
\draw (501.75,306) node [anchor=north west][inner sep=0.75pt]    {$I_{N_{m} +1}$};
\draw (140.75,309) node [anchor=north west][inner sep=0.75pt]    {$\cdots $};
\draw (456.75,306) node [anchor=north west][inner sep=0.75pt]    {$\cdots $};
\draw (244.75,310) node [anchor=north west][inner sep=0.75pt]    {$\cdots $};
\draw (557.75,305) node [anchor=north west][inner sep=0.75pt]    {$\cdots $};

\end{tikzpicture}

    \caption{An example of $p^{R_1,R_2}$ (perturbed around $p_0$) and $\bar{p}^{R_1,R_2}$ (perturbed around $p_1$) when $R_1=\{2,4\}$ and $R_2=\{1,3\}$.}
    \label{fig:perturb}
\end{figure}

We now set $b=1, a_2=1$ and let $a_1 \rightarrow \infty$. For any $p^{R_1,R_2} \in \mathcal{D}_0$, its Hellinger distance from $p_0$ is given by
\begin{align*}
    H^2(p,p_0) &= \frac{1}{2}\left[ (\sqrt{a_1}-0)^2 \left( \frac{1}{2N_m} \right)\left( \frac{1}{a_1}N_m \right) + (\sqrt{2}-1)^2 \left( \frac{1}{2N_m} \right)N_m \right]  \\
    &= \frac{1}{4} \left[ a_1\frac{1}{a_1} + (\sqrt{2}-1)^2 \right] \\
    &= \frac{1}{4} \left[ (\sqrt{2}-1)^2 + 1 \right].
\end{align*}
Also, its Hellinger distance from $p_1$ is given by
\begin{align*}
    H^2(p,p_1) &= \frac{1}{2} \left[ (\sqrt{2}-0)^2 \frac{1}{2N_m} \left(1-\frac{1}{a_1}\right)N_m + (\sqrt{a_1}-\sqrt{2})^2 \frac{1}{2N_m} \frac{1}{a_1} N_m + (\sqrt{1}-0)^2 \frac{1}{2N_m} N_m \right] \\
    &= \frac{1}{4} \left[ 2\left(1-\frac{1}{a_1}\right) + \left(1-\sqrt{\frac{2}{a_1}} \right)^2 + 1 \right]
\end{align*}
We let $a_1 \rightarrow \infty$. Thus, we have
\begin{align*}
    \lim_{a_1 \rightarrow \infty} \frac{H^2(p,p_1)}{H^2(p,p_0)} &= \frac{2+1+1}{4-2\sqrt{2}} \\
    &= \frac{\sqrt{2}}{\sqrt{2}-1}.
\end{align*}
Likewise for any $\bar{p}^{R_1,R_2} \in \mathcal{D}_1$, we have
\begin{align*}
    \lim_{a_1 \rightarrow \infty} \frac{H^2(p,p_0)}{H^2(p,p_1)} &= \frac{\sqrt{2}}{\sqrt{2}-1}.
\end{align*}
Let $\mathcal{D}_0^m$ denote the following product distribution: pick $p$ uniformly at random from the family $\mathcal{D}_0$ and draw m i.i.d. samples from $p$. Likewise, $\mathcal{D}_1^m$ denotes the following product distribution: pick $p$ uniformly at random from the family $\mathcal{D}_1$ and draw m i.i.d. samples from $p$. We now show that $TV(\mathcal{D}_0^m,\mathcal{D}_0^m)$ is small for an appropriate choice of $N_m$. We use \cite[Lemma 24]{bousquet2019optimal}, which states that two distributions are close (in TV) if there exists an event $E$ such that both distributions put a large mass on $E$ and conditioned on this event the two distributions are close (in TV). We reproduce the proof for completeness.
\begin{lemma}[{\cite[Lemma 24]{bousquet2019optimal}}]
    Let $u$ and $v$ be two distributions on a domain $\mathcal{X}$. Let $E \subseteq \mathcal{X}$ be a event. Let $u_{|E}$ and $v_{|E}$ be the respective distributions conditioned on $E$, i.e., $u_{|E}(x) := u(x|E)$, $v_{|E}(x) := v(x|E)$. Then, we have 
    \begin{align*}
        TV(u,v) \le TV(u_{|E},v_{|E}) + 2u(E^c) + 2v(E^c).
    \end{align*}
\end{lemma}
\begin{proof}
    \begin{align*}
        TV(u,v) &= \sup_{A} |u(A)-v(A)| \\
        &= \sup_{A} | u(A \cap E) + u(A \cap E^c) - v(A \cap E) - v(A \cap E^c) | \\
        &\overset{(a)}{\le} \sup_{A} |u(A \cap E) - v(A \cap E)| + \sup_{A} |u(A \cap E^c) - v(A \cap E^c)| \\
        &\overset{(b)}{\le} \sup_{A} |u(E)u(A|E) - v(E)v(A|E) + u(E)v(A|E) - u(E)v(A|E)| + u(E^c) + v(E^c) \\
        &= \sup_{A} |u(E)(u(A|E) - v(A|E))  + v(A|E)(u(E) - v(E))| + u(E^c) + v(E^c) \\
        &\overset{(c)}{\le} \sup_{A} |u(E)(u(A|E) - v(A|E))| + \sup_{A}|v(A|E)(u(E) - v(E))| + u(E^c) + v(E^c) \\
        &\le u(E) \sup_{A} |(u(A|E) - v(A|E))| + |u(E) - v(E)| + u(E^c) + v(E^c) \\
        &\overset{(d)}{\le} TV(u_{|E},v_{|E}) +  |u(E) - v(E)| + u(E^c) + v(E^c) \\
        &= TV(u_{|E},v_{|E}) +  |u(E^c) - v(E^c)| + u(E^c) + v(E^c) \\
        &\le TV(u_{|E},v_{|E}) + 2u(E^c) + 2v(E^c).
    \end{align*}
    $(a)$ and $(c)$ are by triangle inequality, $(b)$ follows from the fact that $u(A \cap E^c) \le u(E^c)$ and $-v(A \cap E^c) \le v(E^c)$, $(d)$ follows from the definition of TV distance.
\end{proof}
We use the above lemma with $u=\mathcal{D}_0^m$, $v=\mathcal{D}_1^m$ and $E$ the event where every interval $I_j$ contains at most one sample. Observe that the event $E$ is invariant under any permutation of $I_j$'s. Furthermore, if $p$ is picked uniformly at random from the family $\mathcal{D}_0$ or $\mathcal{D}_1$ then $\mathbb{E}[p] = \text{unif}[0,1]$, i.e. the uniform distribution over $[0,1]$. Using \cite[Claim 26]{bousquet2019optimal}, we have ${\mathcal{D}_0^m}_{|E}={\mathcal{D}_1^m}_{|E}$. Thus, 
\begin{align*}
    TV(\mathcal{D}_0^m,\mathcal{D}_0^m) \le 2\mathcal{D}_0^m(E^c) + 2\mathcal{D}_1^m(E^c).
\end{align*}
Observe that
\begin{align*}
    \mathcal{D}_0^m(E) &\ge \left( 1-\frac{1}{2N_m} \right) \left( 1-\frac{2}{2N_m} \right) \cdots \left( 1-\frac{m-1}{2N_m} \right) \\
    &\ge \left( 1-\frac{m-1}{2N_m} \right)^{m-1} \\
    &\approx e^{\frac{-(m-1)^2}{2N_m}}
\end{align*}
If we pick $N_m = C(m-1)^2$ then $\mathcal{D}_0^m(E) \ge \frac{11}{12}$ for some appropriate constant $C$. Likewise, we have $\mathcal{D}_1^m(E) \ge \frac{11}{12}$. Thus, we have $TV(\mathcal{D}_0^m,\mathcal{D}_1^m) \le \frac{1}{3}$.
We now use Le Cam's Lemma \cite[Lemma 1]{yu1997assouad}.
\begin{lemma}[{\cite[Lemma 1]{yu1997assouad}}]
    Let $u$ and $v$ be two distributions on the domain $\mathcal{X}$. Consider a test (possibly randomized) that uses $m$ i.i.d. samples from $p \in \{u,v\}$ and decides whether $p=u$ or $p=v$. Then the maximum probability of error of this test is lower bounded by $\frac{1}{2}( 1- TV(u,v))$.
\end{lemma}
\begin{proof}
    A test is given by a set $A \subseteq \mathcal{X}$ where we declare $u$. We declare $v$ if the sample falls in the complementary set  $A^c$. The maximum probability of error is given by
    \begin{align*}
       p(\text{error}) = \max ( u(A^c), v(A) ).
    \end{align*}
    The maximum error probability can be lower bounded by the average error probability (when each distribution is picked with probability 0.5)
    \begin{align*}
        \max (u(A^c), v(A)) &\ge \frac{1}{2}(u(A^c) + v(A)) \\
        &= \frac{1}{2}(1-u(A)+v(A)) \\
        &\ge \frac{1-TV(u,v)}{2}.
    \end{align*}
\end{proof}
We again use this lemma with $u=\mathcal{D}_0^m$, $v=\mathcal{D}_1^m$. Thus, for any test $A \subset \mathcal{X}^m$ we have
\begin{align*}
    \max ( \max_{p \in \mathcal{D}_0} p(A^c), \max_{\bar{p} \in \mathcal{D}_1} \bar{p}(A) ) &\ge \max (\mathcal{D}_0^m(A^c), \mathcal{D}_1^m(A)) \\
    &\ge \frac{1-TV(\mathcal{D}_0^m,\mathcal{D}_1^m)}{2} \\
    &\ge \frac{1}{3}.
\end{align*}
Thus, if a $\gamma$-robust test with $\gamma < \frac{\sqrt{2}}{\sqrt{2}-1}$ existed then we could use it to distinguish between the case where $p \in \mathcal{D}_0$ and where $p \in \mathcal{D}_1$ with probability of error at most $\frac{1}{3}$. This leads to a contradiction and completes the proof.
\end{proof}

\subsection{When is the lower bound tight?}
Let $\text{supp}(p)$ denote the support of distribution $p$, i.e $\text{supp}(p) = \{ x \in \mathcal{X}: p(x)>0 \}$. We now show that our lower bound is tight if $\text{supp}(p_0) \cap \text{supp}(p_1) = \emptyset$. Let $S_1 := \text{supp}(p_0)$, $S_2 := \text{supp}(p_1)$ and $S := \mathcal{X} \backslash (S_1 \cup S_2)$.  In this scenario, any observed sample that is not in either $S_1$ or $S_2$ is not informative for the problem at hand. Thus, a natural test would be to declare $H_0$ (resp. $H_1$) if the fraction of samples falling in $S_1$ (resp. $S_2$) is greater than $\frac{1}{2}$. Let
\begin{align*}
    T(X^m) = \frac{1}{m} \sum_{i=1}^{m} \mathrm{1}[X_i \in S_1] - \mathrm{1}[X_i \in S_2] 
\end{align*}
Thus, we have
\begin{align*}
    \mathbb{E}[T(X^m)] = p(S_1) - p(S_2)
\end{align*}
We would like to show that if $\frac{\sqrt{2}}{\sqrt{2}-1} H^2(p,p_0) \le H^2(p,p_1)$ then $\mathbb{E}[T]>0$ and if $H^2(p,p_0) \le \frac{\sqrt{2}}{\sqrt{2}-1}H^2(p,p_1)$ then $\mathbb{E}[T]<0$. We have
\begin{align*}
    H^2(p,p_0) &= \frac{1}{2} \sum_{x \in \mathcal{X}} (\sqrt{p(x)}-\sqrt{p_0(x)})^2 \\
    &= \frac{1}{2} \sum_{x \in S_1} (\sqrt{p(x)}-\sqrt{p_0(x)})^2 + \sum_{x \in S_2} (\sqrt{p(x)}-\sqrt{p_0(x)})^2 + \sum_{x \in S} (\sqrt{p(x)}-\sqrt{p_0(x)})^2 \\
    &\overset{(a)}{=} \frac{1}{2} \left[ \sum_{x \in S_1} p_0(x) - \sum_{x \in S_1} 2\sqrt{p(x)}\sqrt{p_0(x)} + \sum_{x \in S_1} p(x) + \sum_{x \in S_2} p(x) + \sum_{x \in S} p(x) \right] \\
    &= \frac{1}{2} \left[ 2 - 2\sum_{x \in S_1} \left(\sqrt{p(x)}\sqrt{p_0(x)} \right) \right] \\
    &\overset{(b)}{\ge} 1 - \sum_{x \in S_1} \sqrt{p(x)}  \\
    &\overset{(c)}{\ge} 1 - \sqrt{\sum_{x \in S_1} p(x)} \\
    &= 1 - \sqrt{p(S_1)}
\end{align*}
where $(a)$ follows from the definition of $S_1$, $(b)$ follows from the fact that $\sqrt{p_0(x)} \le 1$ and $(c)$ follows from the concavity of the square root function.
Also,
\begin{align*}
    H^2(p,p_1) &= \frac{1}{2}\sum_{x \in \mathcal{X}} (\sqrt{p(x)}-\sqrt{p_1(x)})^2 \\
    &\le \frac{1}{2} \sum_{x \in \mathcal{X}} p(x) + p_1(x) -2 \sqrt{p(x)}\sqrt{p_1(x)} \\
    &= \frac{1}{2} \sum_{x \in \mathcal{X}} p(x) + p_1(x) \\
    &\le 1.
\end{align*}
Thus, under $H_0$, we have
\begin{align*}
    \frac{\sqrt{2}}{\sqrt{2}-1} \left( 1 - \sqrt{p(S_1)} \right) \le \frac{\sqrt{2}}{\sqrt{2}-1} H^2(p,p_0) \le H^2(p,p_1) \le 1
\end{align*}
Rearranging terms, we get $p(S_1) \ge \frac{1}{2}$ or $\mathbb{E}[T(X^m)] \ge 0$. Similarly, we can show that $\mathbb{E}[T(X^m)] \le 0$ under $H_1$. 

Under the constraint $\text{supp}(p_0) \cap \text{supp}(p_1) = \emptyset$, the optimal slack factor for TV distance goes down from 3 to 2. This indicates that our lower bound might not be tight.

\section{Lower bound on testing with respect to symmetric $\chi^2$ distance}
The symmetric $\chi^2$ distance between two distributions $p_0$ and $p_1$ is defined by
\begin{align*}
    \chi^2(p_0,p_1) &:= \bigg \| \frac{p_0-p_1}{\sqrt{p_0+p_1}} \bigg \|_2^2.
\end{align*}
It can be shown that
\begin{align*}
   \frac{1}{4} \chi^2(p_0,p_1) \le H^2(p_0,p_1) \le \frac{1}{2} \chi^2(p_0,p_1)
\end{align*}
We can study the robust testing problem with symmetric $\chi^2$ distance. It is easy to see that Baraud's test~\eqref{eqn:baraud_test} implies an upper bound of $\frac{2(\sqrt{2}+1)}{\sqrt{2}-1}$ for $\gamma$-robust symmetric $\chi^2$-testing.
\begin{align*}
    \mathbb{E}[T(X)] &\ge (2-\sqrt{2}) H^2(p,p_1) - (2+\sqrt{2}) H^2(p,p_0) \\
    &\ge \frac{2-\sqrt{2}}{4} \chi^2(p,p_1) - \frac{2+\sqrt{2}}{2}  \chi^2(p,p_0) \\
\end{align*}
Thus, $\mathbb{E}[T(X)] \ge 0$ if $\chi^2(p,p_1) \ge \frac{2(\sqrt{2}+1)}{\sqrt{2}-1} \chi^2(p,p_0)$. Likewise, we can show that $\mathbb{E}[T(X)] \le 0$ if $\chi^2(p,p_0) \ge \frac{2(\sqrt{2}+1)}{\sqrt{2}-1} \chi^2(p,p_1)$.
We can get a lower bound of $3$ for robust symmetric $\chi^2$ testing with respect to the symmetric $\chi^2$ distance. 
\blue{\begin{theorem} \label{thrm:robust_lb_chi_squared}
    For every $\gamma < 3$, there is a class $\mathcal{P}=\{p_0,p_1\}$ which is not $\gamma$-robustly testable (with respect to the symmetric $\chi^2$ distance).
\end{theorem}
\begin{proof}
    The construction of the lower bound is exactly the same (our choice of parameters is $b=1, a_2=1, a_1 \rightarrow \infty$. We only need to show the following. For any distribution $p$ in the perturbed family around $p_0$, we have
\begin{align*}
    \chi^2(p,p_0) &= \frac{(a_1-0)^2}{a_1} \left( \frac{1}{2N_m} \right)\left( \frac{1}{a_1}N_m \right) + \frac{(2-1)^2}{2+1} \left( \frac{1}{2N_m} \right)N_m  \\
    &= \frac{1}{2} \left( 1 + \frac{1}{3} \right).
\end{align*}
Also,
\begin{align*}
    \chi^2(p,p_1) &= \frac{(2-0)^2}{2} \frac{1}{2N_m} \left(1-\frac{1}{a_1}\right)N_m + \frac{(a_1-2)^2}{a_1+2} \frac{1}{2N_m} \frac{1}{a_1} N_m + \frac{(1-0)^2}{1} \frac{1}{2N_m} N_m \\
    &= \frac{1}{2} \left( 2\left(1-\frac{1}{a_1}\right) +  \frac{(a_1-2)^2}{a_1(a_1+2)} + 1 \right).
\end{align*}
We let $a_1 \rightarrow \infty$. Thus, we have
\begin{align*}
    \lim_{a_1 \rightarrow \infty} \frac{\chi^2(p,p_1)}{\chi^2(p,p_0)} &= \frac{2+1+1}{\frac{4}{3}} \\
    &= 3.
\end{align*}
Likewise, for any distribution $p$ in the perturbed family around $p_1$. The rest of the proof is similar to the Hellinger case.
\end{proof}}
\section{Alternative Formulation of Robustness} \label{sec:alt_test}
In previous sections, we let the two hypotheses as  $H_0 : \gamma d(p,p_0) < d(p,p_1)$ and $ H_1 : d(p,p_0) > \gamma d(p,p_1)$. An alternative way to define robustness is by letting each hypothesis be a Hellinger ball around a fixed distribution. More concretely, let $p_0$ and $p_1$ be our fixed distributions. The true distribution $p$ could lie in a Hellinger squared ball of radius $r$ around $p_0$ or $p_1$. As before, $p_0$, $p_1$, and $p$ are on some common domain $\mathcal{X}$.
\begin{align*}
    H_0 &: H^2(p,p_0) \le r \\
    H_1 &: H^2(p,p_1) \le r.
\end{align*}
Naturally, the radius $r$ should be such that the two balls do not intersect. Let $r^*$ be the radius when the two balls intersect. We want a test which distinguishes between the two hypotheses as long as $r < r*$. Observe that $H_0$ and $H_1$ are convex sets. Hence, it is known that the approach of working with likelihood ratio tests with respect to the closest pair works \cite[Chapter 7]{gine2021mathematical}. We show that a simple modification of Baraud's test~\eqref{eqn:baraud_test} can distinguish between $H_0$ and $H_1$ as long as $r < r^*$. We now characterize $r^*$. Recall that the Hellinger distance is related to the Bhattacharya distance in the following manner:
\begin{align*}
    H^2(p_0,p_1) &= 1-B(p_0,p_1) \\
    &= 1- \sum_{x \in \mathcal{X}} \sqrt{p_0(x)p_1(x)}.
\end{align*}
Since $\sqrt{p_0}$ and $\sqrt{p_1}$ reside on a unit $\ell_2$ ball, we denote $\cos \theta := B(p_0,p_1)$. A geodesic (shortest path) in the Hellinger distance from $p_0$ and $p_1$ is given by \cite[Chapter 7]{gine2021mathematical}
\begin{align*}
    \sqrt{q_{\phi}} = \frac{\sin{(\theta-\phi)}\sqrt{p_0} + \sin{\phi}\sqrt{p_1}}{\sin{\theta}}.
\end{align*}
It can be verified that $q_{\phi}$ is a valid probability distribution using the fact that $\sin^2{(\theta-\phi)} + \sin^2{\phi} + 2\sin{(\theta-\phi)}\sin{\phi}\cos{\theta} = \sin^2{\theta}$. Here, $\phi \in [0,\theta]$ with $q_0 = p_0$ and $q_{\theta}=p_1$. Thus, we can define $u$ to be the "Hellinger midpoint" of $p_0$ and $p_1$, that is,
\begin{align*}
    \sqrt{u} &:= \sqrt{q_{\frac{\theta}{2}}} \\
    &=\frac{\sin{\frac{\theta}{2}}}{\sin{\theta}} (\sqrt{p_0} + \sqrt{p_1}).
\end{align*}
Thus,
\begin{align*}
    r^* &= H^2(p_0,u) \\
    &= 1 -  \sum_{x \in \mathcal{X}} \sqrt{p_0(x)} \sqrt{u(x)} \\
    &= 1 - \cos \frac{\theta}{2}. 
\end{align*}
Define $T$ to be the following test statistic (similar to Baraud but using the Hellinger midpoint rather than the Euclidean midpoint).
\begin{equation} \label{eqn:alt_test}
    T(X^m) := \frac{1}{m} \sum_{i=1}^{m} \left( \sqrt{\frac{p_0(X_i)}{u(X_i)}} - \sqrt{\frac{p_1(X_i)}{u(X_i)}} \right)
\end{equation}

\begin{theorem} \label{lem:alt_test}
    The test described in \eqref{eqn:alt_test} can distinguish between $H_0$ and $H_1$ as long as $r < 1 - \cos \frac{\theta}{2}$ where $\theta = \cos^{-1} B(p_0, p_1)$.
\end{theorem}
\begin{proof}
   We analyze the expected value of $T(X^m)$ with respect to the unknown distribution $p$. By linearity of expectation, it suffices to analyze for $m=1$.
\begin{align*}
    E[T(X)] = \sum_{x \in \mathcal{X}} p(x) \left( \sqrt{\frac{p_0(x)}{u(x)}} - \sqrt{\frac{p_1(x)}{u(x)}} \right).
\end{align*}
We now add and subtract $2\sum \sqrt{p(x)p_0(x)}$, $2\sum \sqrt{p(x)p_1(x)}$ and $\sum\sqrt{p_0(x)u(x)}$ (the last term is also equal to  $\sum\sqrt{p_1(x)u(x)}$ by our choice of $u$). Thus, we have
\begin{align*}
    &E[T(X)] \\
    &= 2\sum_{x \in \mathcal{X}} \sqrt{p(x)p_0(x)} - 2\sum_{x \in \mathcal{X}} \sqrt{p(x)p_1(x)} + \sum_{x \in \mathcal{X}} \frac{\sqrt{p_0(x)}}{\sqrt{u(x)}}(\sqrt{p(x)}-\sqrt{u(x)})^2 - \sum_{x \in \mathcal{X}} \frac{\sqrt{p_1(x)}}{\sqrt{u(x)}}(\sqrt{p(x)}-\sqrt{u(x)})^2 \\
    &= 2\sum_{x \in \mathcal{X}} \sqrt{p(x)p_0(x)} - 2\sum_{x \in \mathcal{X}} \sqrt{p(x)p_1(x)} + \sum_{x \in \mathcal{X}} \frac{\sqrt{p_0(x)}-\sqrt{p_1(x)}}{\sqrt{u(x)}}(\sqrt{p(x)}-\sqrt{u(x)})^2 \\
    &= 2\sum_{x \in \mathcal{X}} \sqrt{p(x)p_0(x)} - 2\sum_{x \in \mathcal{X}} \sqrt{p(x)p_1(x)} + \sum_{x \in \mathcal{X}} \frac{\sqrt{p_0(x)}+\sqrt{p_1(x)}-2\sqrt{p_1(x)}}{\sqrt{u(x)}}(\sqrt{p(x)}-\sqrt{u(x)})^2 \\
    &\overset{(a)}{=} 2\sum_{x \in \mathcal{X}} \sqrt{p(x)p_0(x)} - 2\sum_{x \in \mathcal{X}} \sqrt{p(x)p_1(x)} + \frac{\sin{\theta}}{\sin{\frac{\theta}{2}}} \sum_{x \in \mathcal{X}} \frac{\sqrt{p_0(x)}+\sqrt{p_1(x)}-2\sqrt{p_1(x)}}{\sqrt{p_0(x)}+\sqrt{p_1(x)}}(\sqrt{p(x)}-\sqrt{u(x)})^2 \\
    &\overset{(b)}{\ge} 2\sum_{x \in \mathcal{X}} \sqrt{p(x)p_0(x)} - 2\sum_{x \in \mathcal{X}} \sqrt{p(x)p_1(x)} - \frac{\sin{\theta}}{\sin{\frac{\theta}{2}}} \sum_{x \in \mathcal{X}} (\sqrt{p(x)}-\sqrt{u(x)})^2 \\
    &= 2\sum_{x \in \mathcal{X}} \sqrt{p(x)p_0(x)} - 2\sum_{x \in \mathcal{X}} \sqrt{p(x)p_1(x)} - 2\frac{\sin{\theta}}{\sin{\frac{\theta}{2}}} (1-\sum_{x \in \mathcal{X}}\sqrt{p(x)u(x)}) \\
    &= 2\sum_{x \in \mathcal{X}} \sqrt{p(x)p_0(x)} - 2\sum_{x \in \mathcal{X}} \sqrt{p(x)p_1(x)} - 2\frac{\sin{\theta}}{\sin{\frac{\theta}{2}}} \left( 1-\sum_{x \in \mathcal{X}}\sqrt{p(x)} \frac{\sin{\frac{\theta}{2}}}{\sin{\theta}}(\sqrt{p_0(x)}+\sqrt{p_1(x)}) \right) \\
    &= 4\sum_{x \in \mathcal{X}} \sqrt{p(x)p_0(x)} - 2\frac{\sin{\theta}}{\sin{\frac{\theta}{2}}} \\
    &= 4 \left( \sum_{x \in \mathcal{X}} \sqrt{p(x)p_0(x)} - \cos{\frac{\theta}{2}} \right)
\end{align*}
$(a)$ follows from the definition of $\sqrt{u}$, $(b)$ follows from the fact that $\frac{\sqrt{p_1(x)}}{\sqrt{p_0(x)}+\sqrt{p_1(x)}} \le 1$. Thus, if $\sum_{x \in \mathcal{X}} \sqrt{p(x)p_0(x)} > \cos{\frac{\theta}{2}}$, $E[T(X)]>0$. If $\sum \sqrt{pp_1} > \cos{\frac{\theta}{2}}$, $E[T(X)]<0$. Or, in other words, if $H^2(p,p_0) < 1 - \cos{\frac{\theta}{2}}$, then $E[T(X)] > 0$. If $H^2(p,p_1) < 1 - \cos{\frac{\theta}{2}}$, then $E[T(X)] < 0$. Thus, our test can distinguish between $H_0$ and $H_1$ as long as $r < r^* = 1 - \cos{\frac{\theta}{2}}$. 
\end{proof}
 
\section{Discussion}
The problem of exactly characterizing the optimal slack factor $\gamma^*$ remains open. When the distance metric is TV, it is known that a randomized test can reduce $\gamma^*$ from 3 to 2 \cite{mahalanabis2007density}. It would be interesting to see an analogous result for the Hellinger distance. Finally, quantum analogoues of this problem (with distributions replaced with quantum states) is a direction we would love to explore.

\section{Acknowledgments}
We thank Ankit Pensia for bringing to our attention the following observation. Let $d(.,.)$ be any distance metric. Consider a different formulation of the robust testing problem. $H_0': d(p,p_0) \le \frac{d(p_0,p_1)}{C}$ vs $H_1': d(p,p_1) \le \frac{d(p_0,p_1)}{C}$. Birge (cite?) gives a test that works for any $C>2$. This implies a test for the problem $H_0: (C+1)d(p,p_0) \le d(p,p_1)$ vs $H_1: d(p,p_0) \ge (C+1)d(p,p_1)$. Observe that
\begin{align*}
    (C+1)d(p,p_0) &\le d(p,p_1) \le d(p,p_0) + d(p_0,p_1) \\
    C d(p,p_0) &\le d(p_0,p_1).
\end{align*}
Thus, $H_0 \implies H_0'$. Likewise, $H_1 \implies H_1'$. Since Birge's test works for any $C>2$, it also works for any $C+1>3$ for our problem. In particular, for the Hellinger distance, this bound is better than the one in \cite{suresh2021robust} but not as good as the one in \cite{baraud2011estimator}.

\bibliographystyle{ieeetr}

\bibliography{refs}

\end{document}